\documentclass[12pt]{article}%
\usepackage{amsmath}
\usepackage{amsfonts}
\usepackage{amssymb}
\usepackage{graphicx}%
\providecommand{\U}[1]{\protect\rule{.1in}{.1in}}
\newtheorem{theorem}{Theorem}

\newtheorem{algorithm}[theorem]{Algorithm}

\newtheorem{corollary}[theorem]{Corollary}

\newtheorem{lemma}[theorem]{Lemma}

\newenvironment{proof}[1][Proof]{\textbf{#1.} }{\ \rule{0.5em}{0.5em}}

\date{}

\begin{document}

\title{Simultaneous Unitary Equivalences\thanks{Linear Algebra Appl. (2011), doi:10.1016/j.laa.2011.09.031}}
\author{Tatiana G. Gerasimova\thanks{Faculty of Mechanics and Mathematics, Kiev
National Taras Shevchenko University,
Volodymyrska 64, Kiev, Ukraine.
\texttt{gerasimova@imath.kiev.ua}},
Roger A. Horn\thanks{Mathematics
Department, University of Utah, Salt
Lake City, Utah USA.
\texttt{rhorn@math.utah.edu}}, and
\\ Vladimir V.
Sergeichuk\thanks{Institute of
Mathematics, Tereshchenkivska 3, Kiev,
Ukraine. Supported in part by Grant
0107U002333 from the National Academy
of Sciences of Ukraine.
\texttt{sergeich@imath.kiev.ua}}}
\maketitle
\date{\it Dedicated with respect and appreciation to Avi Berman, Moshe Goldberg, and
Raphael Loewy on the occasion of their retirement from The Technion.}

\begin{abstract}
Let $\mathcal{S}_{1}$, $\mathcal{S}_{2}$, $\mathcal{S}_{3}$, $\mathcal{S}_{4}$
be given finite sets of pairs of $n$-by-$n$ complex matrices. We describe an
algorithm to determine, with finitely many computations, whether there is a
single unitary matrix $U$ such that each pair of matrices in $\mathcal{S}_{1}$
is unitarily similar via $U$, each pair of matrices in $\mathcal{S}_{2}$ is
unitarily congruent via $U$, each pair of matrices in $\mathcal{S}_{3}$ is
unitarily similar via $\bar{U}$, and each pair of matrices in $\mathcal{S}%
_{4}$ is unitarily congruent via $\bar{U}$.
\medskip

\medskip

\end{abstract}

\section{Introduction}

Our goal is to solve the following

\medskip%
%TCIMACRO{\TeXButton{Noindent}{\noindent}}%
%BeginExpansion
\noindent
%EndExpansion
\textbf{General Problem.} Let $\mathcal{S}_{1}$, $\mathcal{S}_{2}$,
$\mathcal{S}_{3}$, $\mathcal{S}_{4}$ be given finite sets of pairs of
$n$-by-$n$ complex matrices. Describe an algorithm to determine, with finitely
many computations, whether there is a single unitary matrix $U$ such that each
pair of matrices in $\mathcal{S}_{1}$ is unitarily similar via $U$, each pair
of matrices in $\mathcal{S}_{2}$ is unitarily congruent via $U$, each pair of
matrices in $\mathcal{S}_{3}$ is unitarily similar via $\bar{U}$, and each
pair of matrices in $\mathcal{S}_{4}$ is unitarily congruent via $\bar{U}$.

\medskip This General Problem includes as special cases the problem of
determining whether finitely many pairs of matrices are simultaneously
unitarily similar ($\mathcal{S}_{2}=\mathcal{S}_{3}=\mathcal{S}_{4}%
=\varnothing$) as well as the problem of determining whether finitely many
pairs of matrices are simultaneously unitarily congruent ($\mathcal{S}%
_{1}=\mathcal{S}_{3}=\mathcal{S}_{4}=\varnothing$).

All of the matrices that we consider are complex and square. Two matrices $A$
and $B$ of the same size are \emph{unitarily similar} if there is a unitary
matrix $U$ such that $A=UBU^{\ast}$; they are \emph{unitarily congruent} if
there is a unitary matrix $U$ such that $A=UBU^{T}$. Given pairs of $n$-by-$n$
matrices $(A_{1},B_{1}),\ldots,(A_{m},B_{m})$ are \emph{simultaneously
unitarily similar} if there is a unitary matrix $U$ such that $A_{j}%
=UB_{j}U^{\ast}$ for each
$j=1,\ldots,m$; they are
\emph{simultaneously unitarily
congruent }if there is a unitary
matrix\emph{ }$U$ such that
$A_{j}=UB_{j}U^{T}$ for each
$j=1,\ldots,m$. The trace of a matrix
$A$ is denoted by $\operatorname{tr}A$.
We adopt the notation and terminology
of \cite{HJ1}.

\section{Unitary similarity of a pair of matrices}

Any finite formal product of nonnegative powers of two noncommuting variables
$s$, $t$
\[
W(s,t)=s^{m_{1}}t^{n_{1}}s^{m_{2}}t^{n_{2}}\cdots s^{m_{k}}t^{n_{k}},\quad
m_{1},n_{1},\ldots,m_{k},n_{k}\geq0
\]
is a \emph{word }in\emph{ }$s$ and $t$. The sum $m_{1}+n_{1}+m_{2}%
+n_{2}+\cdots+m_{k}+n_{k}$ is the \emph{length} of the word ${W}(s,t)$, and
the nonnegative integers $m_{i}$ and $n_{i}$ are its \emph{factor exponents}.
A \emph{word in }$A$\emph{ and }$A^{\ast}$ is%

\begin{equation}
W(A,A^{\ast})=A^{m_{1}}(A^{\ast})^{n_{1}}A^{m_{2}}(A^{\ast})^{n_{2}}\cdots
A^{m_{k}}(A^{\ast})^{n_{k}} \label{word}%
\end{equation}

If $A=UBU^{\ast}$ for some unitary $U\in M_{n}$, a calculation reveals that
$W(A,A^{\ast})=UW(B,B^{\ast})U^{\ast}$, so $W(A,A^{\ast})$ is unitarily
similar to $W(B,B^{\ast})$. Thus, unitary similarity of $A$ and $B$ implies
that
\begin{equation}
\operatorname{tr}W(A,A^{\ast})=\mathrm{\operatorname{tr}}W(B,B^{\ast})
\label{trace}%
\end{equation}
for every word $W(s,t)$ in two noncommuting variables.

A theorem of W. Specht \cite{Specht}
provides a converse for this
implication. Various authors have
provided bounds to show that only
finitely many words need to be
considered \cite{P}; the bound in the
following theorem is due to Pappacena.
\cite{Pappacena}

\begin{theorem}
\label{Specht}Let complex $n$-by-$n$ matrices $A$ and $B$ be given. The
following are equivalent:
%TCIMACRO{\TeXButton{Newline}{\newline}}%
%BeginExpansion
\newline
%EndExpansion
(a) $A$ and $B$ are unitarily similar;
%TCIMACRO{\TeXButton{Newline}{\newline}}%
%BeginExpansion
\newline
%EndExpansion
(b) $\operatorname{tr}W(A,A^{\ast})=\operatorname{tr}W(B,B^{\ast})$ for every
word ${W}(s,t)$ in two noncommuting variables;
%TCIMACRO{\TeXButton{Newline}{\newline}}%
%BeginExpansion
\newline
%EndExpansion
(c) $\operatorname{tr}W(A,A^{\ast})=\operatorname{tr}W(B,B^{\ast})$ for every
word ${W}(s,t)$ in two noncommuting variables whose length is at most%
\[
n\sqrt{\frac{2n^{2}}{n-1}+\frac{1}{4}}+\frac{n}{2}-2\text{.}%
\]

\end{theorem}

\section{A basic lemma}

The key to obtaining our criteria for simultaneous unitary similarity and
simultaneous unitary congruence is understanding the consequences of certain
intertwining relations involving a special block matrix.

\begin{lemma}
\label{basic} Consider the complex $k$-by-$k$ block matrices%
\begin{equation}
A=\left[
\begin{array}
[c]{ccccc}%
0 & I_{n} & A_{1,3} & \cdots & A_{1,k}\\
& 0 & I_{n} & \ddots & \vdots\\
&  & 0 & \ddots & A_{k-2,k}\\
&  &  & \ddots & I_{n}\\
&  &  &  & 0
\end{array}
\right]  \label{Ma}%
\end{equation}
and%
\begin{equation}
B=\left[
\begin{array}
[c]{ccccc}%
0 & I_{n} & B_{1,3} & \cdots & B_{1,k}\\
& 0 & I_{n} & \ddots & \vdots\\
&  & 0 & \ddots & B_{k-2,k}\\
&  &  & \ddots & I_{n}\\
&  &  &  & 0
\end{array}
\right]  \label{Mb}%
\end{equation}
in which every block is $n$-by-$n$. Define $A_{i,i+1}=B_{i,i+1}=I_{n}$ for all
$i=1,\ldots,k-1$ and $A_{ij}=B_{ij}=0$ whenever $i\geq j$. Let $W=[W_{ij}%
]_{i,j=1}^{k}$ be an $nk$-by-$nk$ matrix that is partitioned conformally to
$A$ and $B$.
%TCIMACRO{\TeXButton{Newline}{\newline}}%
%BeginExpansion
\newline
%EndExpansion
(a) Suppose that $AW=WB$. Then $W$ is block upper triangular and
$W_{11}=W_{22}=\cdots=W_{kk}$.
%TCIMACRO{\TeXButton{Newline}{\newline}}%
%BeginExpansion
\newline
%EndExpansion
(b) Suppose that $W$ is unitary and $AW=WB$, that is, $A$ and $B$ are
unitarily similar and $A=WBW^{\ast}$. Then $W_{11}=U$ is unitary and
$W=U\oplus U\oplus\cdots\oplus U$ is block diagonal. Moreover, $A_{ij}%
=UB_{ij}U^{\ast}$ for all $i$ and $j$.
%TCIMACRO{\TeXButton{Newline}{\newline}}%
%BeginExpansion
\newline
%EndExpansion
(c) Suppose that $A\bar{W}=WB$. Then $W$ is block upper triangular,
$W_{ii}=W_{11}$ if $i$ is odd, and $W_{ii}=\overline{W_{11}}$ if $i$ is even.
%TCIMACRO{\TeXButton{Newline}{\newline}}%
%BeginExpansion
\newline
%EndExpansion
(d) Suppose that $W$ is unitary and $A\bar{W}=WB$, that is, $A$ is unitarily
congruent to $B$ and $A=WBW^{T}$. Then $W_{11}=U$ is unitary, $W_{ii}=U$ if
$i$ is odd, $W_{ii}=\bar{U}$ if $i$ is even, and $W=U\oplus\bar{U}\oplus
U\oplus\cdots$ is block diagonal. Moreover,
%TCIMACRO{\TeXButton{Newline}{\newline}}%
%BeginExpansion
\newline
%EndExpansion%
%TCIMACRO{\TeXButton{indent}{\indent}}%
%BeginExpansion
\indent
%EndExpansion
\noindent(d1) $A_{ij}=UB_{ij}U^{\ast}$ if $i$ is odd and $j$ is even;
%TCIMACRO{\TeXButton{Newline}{\newline}}%
%BeginExpansion
\newline
%EndExpansion%
%TCIMACRO{\TeXButton{indent}{\indent}}%
%BeginExpansion
\indent
%EndExpansion
\noindent(d2) $A_{ij}=UB_{ij}U^{T}$ if $i$ and $j$ are both odd;
%TCIMACRO{\TeXButton{Newline}{\newline}}%
%BeginExpansion
\newline
%EndExpansion%
%TCIMACRO{\TeXButton{indent}{\indent}}%
%BeginExpansion
\indent
%EndExpansion
\noindent(d3) $A_{ij}=\bar{U}B_{ij}U^{\ast}$ if $i$ and $j$ are both even;
and
%TCIMACRO{\TeXButton{Newline}{\newline}}%
%BeginExpansion
\newline
%EndExpansion%
%TCIMACRO{\TeXButton{indent}{\indent}}%
%BeginExpansion
\indent
%EndExpansion
\noindent(d4) $A_{ij}=\bar{U}B_{ij}U^{T}$ if $i$ is even and $j$ is odd.
\end{lemma}

\begin{proof}
A computation verifies the assertions in (a) and (c) about block
triangularity: compare blocks in the respective identities $AW=WB$ and
$A\bar{W}=WB$, starting in block position ($k,1$). Work to the right until
reaching block position ($k,k-1$). Move up to block position ($k-1,1$) and
work to the right until reaching block position ($k-1,k-2$). Repeat this
process, moving up one block row at a time, until reaching block position
($2,1$).

The assertions in (a) and (c) about the main diagonal blocks of $W$ follow
(once one knows that $W$ is block upper triangular) from comparing blocks in
the respective identities in positions ($1,2$),\ldots,($k-1,k$).

The assertions in (b) and (d) about $W$ reflect the facts that a block
triangular unitary matrix is block diagonal, and the direct summands in a
unitary direct sum are unitary. The asserted relationships between $A_{ij}$
and $B_{ij}$ follow (once one knows that $W$ is block diagonal) from the
respective identities $A_{ij}W_{jj}=W_{ii}A_{ij}$ and $A_{ij}\overline{W_{jj}%
}=W_{ii}A_{ij}$.\hfill
\end{proof}

\medskip In summary, if the matrices $A$ and $B$ in (\ref{Ma}) and (\ref{Mb})
are unitarily similar then all of the pairs $(A_{ij},B_{ij})$ are
simultaneously unitarily similar. If $A$ and $B$ are unitarily congruent, then
there is a single unitary matrix $U$ involved in four types of unitary
equivalence: certain pairs $(A_{ij},B_{ij})$ are simultaneously unitarily
similar via $U$ or $\bar{U}$, and certain pairs are simultaneously unitarily
congruent via $U$ or $\bar{U}$.

\section{Simultaneous unitary similarity}

It is useful to have an explicit statement of the criterion for simultaneous
unitary similarity that is implicit in Lemma \ref{basic}.

\begin{theorem}
\label{simul unitary}Let pairs $(A_{1},B_{1}),\ldots,(A_{m},B_{m})$ of
$n$-by-$n$ complex matrices be given. Choose $k$ large enough so that the
matrix $A$ in (\ref{Ma}) has at least $m$ blocks above the second block
superdiagonal. Place the matrices $A_{1},\ldots,A_{m}$ in those blocks in any
order, and place zero matrices in any unfilled blocks. Place $B_{1}%
,\ldots,B_{m}$ and zero matrices in corresponding blocks of the matrix $B$ in
(\ref{Mb}). Then $A$ is unitarily similar to $B$, if and only if the pairs
$(A_{1},B_{1}),\ldots,(A_{m},B_{m})$ are simultaneously unitarily similar.
\end{theorem}

For example, one could choose $k=m+2$ and place the respective matrices of the
pairs $(A_{1},B_{1}),\ldots,(A_{m},B_{m})$ in positions $(1,3),\ldots,(1,k)$
of the first block rows of (\ref{Ma}) and (\ref{Mb}), or in the blocks
$(1,3),(2,4),\ldots,(k-2,k)$ of the third block superdiagonal of (\ref{Ma})
and (\ref{Mb}).

A natural extension of Specht's criterion to more than a single pair of
matrices follows from the preceding theorem.

\begin{corollary}
\label{genSpecht}Given pairs $(A_{1},B_{1}),\ldots,(A_{m},B_{m})$ of
$n$-by-$n$ complex matrices are simultaneously unitarily similar if and only
if $\operatorname{tr}w(A_{1},A_{1}^{\ast},\ldots,A_{m},A_{m}^{\ast
})=\operatorname{tr}w(B_{1},B_{1}^{\ast},\ldots,B_{m},B_{m}^{\ast})$ for all
words $w(s_{1},t_{1},\ldots,s_{m},t_{m})$ in $2m$ noncommuting variables.
\end{corollary}

\begin{proof}
Let $k=m+2$ and consider a matrix $A$ of the form (\ref{Ma}) that is
constructed by placing the matrices $A_{1},\ldots,A_{m}$ sequentially in the
$m$ blocks of its third block superdiagonal, and placing zero blocks in all of
its other blocks above the third block superdiagonal. Construct a matrix $B$
of the form (\ref{Mb}) in the same way using the matrices $B_{1},\ldots,B_{m}$.

Consider the following assertions:
%TCIMACRO{\TeXButton{Newline}{\newline}}%
%BeginExpansion
\newline
%EndExpansion
(a)\ \
$\operatorname{tr}w(A_{1},A_{1}^{\ast},\ldots,A_{m},A_{m}^{\ast
})\ =\
\operatorname{tr}w(B_{1},B_{1}^{\ast},\ldots,B_{m},B_{m}^{\ast})$
\ \ \ for\ \ \ all\ \ \  words
\\ $w(s_{1},t_{1},\ldots,s_{m},t_{m})$ in
$2m$ noncommuting variables;
%TCIMACRO{\TeXButton{Newline}{\newline}}%
%BeginExpansion
\newline
%EndExpansion
(b) $\operatorname{tr}W(A,A^{\ast})=\operatorname{tr}W(B,B^{\ast})$ for every
word $W(s,t)$ in two noncommuting variables;
%TCIMACRO{\TeXButton{Newline}{\newline}}%
%BeginExpansion
\newline
%EndExpansion
(c) $A$ is unitarily similar to $B$;
%TCIMACRO{\TeXButton{Newline}{\newline}}%
%BeginExpansion
\newline
%EndExpansion
(d) The pairs $(A_{1},B_{1}),\ldots,(A_{m},B_{m})$ are simultaneously
unitarily similar.

It suffices to show that these four assertions are equivalent.
%TCIMACRO{\TeXButton{Newline}{\newline}}%
%BeginExpansion
\newline
%EndExpansion
(a) $\Rightarrow$ (b) Each block of any word $W(A,A^{\ast})$ is a linear
combination of the identity matrix (we may think of it as an empty word) and
words of the form $w(A_{1},A_{1}^{\ast},\ldots,A_{m},A_{m}^{\ast})$.
%TCIMACRO{\TeXButton{Newline}{\newline}}%
%BeginExpansion
\newline
%EndExpansion
(b) $\Rightarrow$ (c) Theorem \ref{Specht}.
%TCIMACRO{\TeXButton{Newline}{\newline}}%
%BeginExpansion
\newline
%EndExpansion
(c) $\Rightarrow$ (d) Theorem \ref{simul unitary}.
%TCIMACRO{\TeXButton{Newline}{\newline}}%
%BeginExpansion
\newline
%EndExpansion
(d) $\Rightarrow$ (a) The same computation that verified the identities
(\ref{trace}).\hfill
\end{proof}

\medskip Simultaneous unitary similarity of given pairs $(A_{1},B_{1}%
),\ldots,(A_{m},B_{m})$ of $n$-by-$n$ complex matrices is equivalent to
unitary similarity of two particular block matrices; Theorem \ref{Specht}
ensures that this latter unitary similarity can be confirmed or refuted with
finitely many computations. Thus, there is a finite algorithm to determine
whether a finite number of pairs of matrices are simultaneously unitarily similar.

Likewise, the trace criterion in Corollary \ref{genSpecht} requires only
finitely many computations: it requires verification only of enough identities
of the form $\operatorname{tr}w(A_{1},A_{1}^{\ast},\ldots,A_{m},A_{m}^{\ast
})=$ $\operatorname{tr}w(B_{1},B_{1}^{\ast},\ldots,B_{m},B_{m}^{\ast})$ to
ensure satisfaction of the finite number of identities of the form
$\operatorname{tr}W(A,A^{\ast})=\operatorname{tr}W(B,B^{\ast})$ required by
the Pappacena upper bound in Theorem \ref{Specht}.

\section{Unitary congruence of a pair of matrices}

Before we consider the role of Lemma \ref{basic} in assessing simultaneous
unitary congruence and other simultaneous unitary equivalences, we need to
consider the simplest case of unitary congruence of a single pair of matrices.
If $U$ is unitary and $A=UBU^{T}$, a calculation reveals that
\begin{equation}
AA^{\ast}=U(BB^{\ast})U^{\ast}\text{, }A\bar{A}=U(B\bar{B})U^{\ast}\text{, and
}A^{T}\bar{A}=U(B^{T}\bar{B})U^{\ast}\text{,} \label{basic three}%
\end{equation}
so three pairs of matrices related to $A$ and $B$ are simultaneously unitarily
similar. Fortunately, this necessary condition is also sufficient:

\begin{theorem}
\label{HH} Complex $n$-by-$n$ matrices $A$ and $B$ are unitarily congruent if
and only if the three pairs $(AA^{\ast},BB^{\ast})$, $(A\bar{A},B\bar{B})$,
and $(A^{T}\bar{A},B^{T}\bar{B})$ are simultaneously unitarily similar. If
either $A$ or $B$ is nonsingular, the third pair may be omitted.
\end{theorem}

\begin{proof}
See \cite{HornHong}.\hfill
\end{proof}

\medskip The preceding theorem and Theorem \ref{simul unitary} imply the
following criterion for unitary congruence.

\begin{theorem}
\label{triple} Complex $n$-by-$n$ matrices $A$ and $B$ are unitarily congruent
if and only if the $4n$-by-$4n$ matrices%
\begin{equation}
K_{A}=\left[
\begin{array}
[c]{cccc}%
0 & I_{n} & AA^{\ast} & A\bar{A}\\
& 0 & I_{n} & A^{T}\bar{A}\\
&  & 0 & I_{n}\\
&  &  & 0
\end{array}
\right]  \text{ and }K_{B}=\left[
\begin{array}
[c]{cccc}%
0 & I_{n} & BB^{\ast} & B\bar{B}\\
& 0 & I_{n} & B^{T}\bar{B}\\
&  & 0 & I_{n}\\
&  &  & 0
\end{array}
\right]  \label{basic two}%
\end{equation}
are unitarily similar. If $A$ or $B$ is
nonsingular, they are unitarily
congruent if and only if%
\begin{equation}
K_{A}^{\prime}=\left[
\begin{array}
[c]{cccc}%
0 & I_{n} & AA^{\ast} & A\bar{A}\\
& 0 & I_{n} & 0\\
&  & 0 & I_{n}\\
&  &  & 0
\end{array}
\right]  \text{ and }K_{B}^{\prime}=\left[
\begin{array}
[c]{cccc}%
0 & I_{n} & BB^{\ast} & B\bar{B}\\
& 0 & I_{n} & 0\\
&  & 0 & I_{n}\\
&  &  & 0
\end{array}
\right]  \label{basic two'}%
\end{equation}
are unitarily similar.
\end{theorem}

When one applies the criterion in Theorem \ref{triple}, the bound in Theorem
\ref{Specht} ensures that it suffices to verify identities of the form
\begin{equation}
\operatorname{tr}W(K_{A},K_{A}^{\ast})=\operatorname{tr}W(K_{B},K_{B}^{\ast})
\label{Ktrace}%
\end{equation}
for all words $W(s,t)$ of length at most%
\begin{equation}
4n\sqrt{\frac{32n^{2}}{4n-1}+\frac{1}{4}}+2n-2\text{.} \label{sqrt}%
\end{equation}
The matrices (\ref{basic two}) and (\ref{basic two'}) are nilpotent of index
four, so in verifying the trace identities (\ref{Ktrace}) we need to consider
only words, all of whose factor exponents are three or less.

For $n=2$, the bound (\ref{sqrt}) says that it suffices to consider all words
of length at most $37$, so a great many words must be considered, even in the
smallest case. For $n=3$ the upper bound on the length is $66$; for $n=4$ it
is $116$. In contrast, when employing Specht's criterion to test a pair of
2-by-2 matrices for unitary similarity, it suffices to check traces of only
three words of length at most two; to test a pair of 3-by-3 matrices it
suffices to check traces of only seven words of length at most six. It is not
known whether the special form of the matrices in (\ref{basic two}) and
(\ref{basic two'}), some special features in low dimensional cases, or some
clever insight would permit the upper bound (\ref{sqrt}) to be reduced significantly.

\section{Simultaneous unitary equivalences}

We now make explicit the solution of the General Problem that is implicit in
Lemma \ref{basic}.

\begin{theorem}
Let $m_{1}$, $m_{2}$, $m_{3}$, and $m_{4}$ be given nonnegative integers. Let%
\[%
\begin{array}
[c]{ccc}%
(A_{1}^{(1)},B_{1}^{(1)}),\ldots,(A_{m_{1}}^{(1)},B_{m_{1}}^{(1)}), & \quad &
(A_{1}^{(2)},B_{1}^{(2)}),\ldots,(A_{m_{2}}^{(2)},B_{m_{2}}^{(2)}),\\
(A_{1}^{(3)},B_{1}^{(3)}),\allowbreak\ldots,\allowbreak(A_{m_{3}}%
^{(3)},B_{m_{3}}^{(3)}), & \quad & (A_{1}^{(4)},B_{1}^{(4)}),\ldots,(A_{m_{4}%
}^{(4)},B_{m_{4}}^{(4)})
\end{array}
\]
be given pairs of $n$-by-$n$ complex matrices. Choose $k$ large enough so that
the matrix $A$ in (\ref{Ma}) has enough blocks above the second block
superdiagonal to accommodate the following construction:
%TCIMACRO{\TeXButton{Newline}{\newline}}%
%BeginExpansion
\newline
%EndExpansion
(1) Place the matrices $A_{1}^{(1)},\ldots,A_{m_{1}}^{(1)}$ (in any desired
order) in ($i,j$) blocks of $A$ such that $i$ is odd, $j$ is even, and
$j-i\geq2$; place the matrices $B_{1}^{(1)},\ldots,B_{m_{1}}^{(1)}$ in
corresponding positions in $B$.
%TCIMACRO{\TeXButton{Newline}{\newline}}%
%BeginExpansion
\newline
%EndExpansion
(2) Place $A_{1}^{(2)},\ldots,A_{m_{2}}^{(2)}$ (in any desired order) in
($i,j$) blocks of $A$ such that $i$ and $j$ are both odd and $j-i\geq2$; place
$B_{1}^{(2)},\ldots,B_{m_{2}}^{(2)}$ in corresponding positions in $B$.
%TCIMACRO{\TeXButton{Newline}{\newline}}%
%BeginExpansion
\newline
%EndExpansion
(3) Place $A_{1}^{(3)},\ldots,A_{m_{3}}^{(3)}$ (in any desired order) in
($i,j$) blocks of $A$ such that $i$ and $j$ are both even and $j-i\geq2$;
place $B_{1}^{(3)},\ldots,B_{m_{3}}^{(3)}$ in corresponding positions in $B$.
%TCIMACRO{\TeXButton{Newline}{\newline}}%
%BeginExpansion
\newline
%EndExpansion
(4) Place $A_{1}^{(4)},\ldots,A_{m_{4}}^{(4)}$ (in any desired order) in
($i,j$) blocks of $A$ such that $i$ is even, $j$ is odd, and $j-i\geq2$; place
$B_{1}^{(4)},\ldots,B_{m4}^{(4)}$ in corresponding positions in $B$.
%TCIMACRO{\TeXButton{Newline}{\newline}}%
%BeginExpansion
\newline
%EndExpansion
(5) Place zero matrices in any unfilled blocks of $A$ and $B$.
%TCIMACRO{\TeXButton{Newline}{\newline}}%
%BeginExpansion
\newline
%EndExpansion%
%TCIMACRO{\TeXButton{indent}{\indent}}%
%BeginExpansion
\indent
%EndExpansion
\noindent Then $A$ is unitarily congruent to $B$ if and only if there is a
unitary matrix $U$ such that
%TCIMACRO{\TeXButton{Newline}{\newline}}%
%BeginExpansion
\newline
%EndExpansion%
%TCIMACRO{\TeXButton{indent}{\indent}}%
%BeginExpansion
\indent
%EndExpansion
($1^{\prime}$) $A_{i}^{(1)}=UB_{i}^{(1)}U^{\ast}$ for all $i=1,\ldots,m_{1}$;
%TCIMACRO{\TeXButton{Newline}{\newline}}%
%BeginExpansion
\newline
%EndExpansion%
%TCIMACRO{\TeXButton{indent}{\indent}}%
%BeginExpansion
\indent
%EndExpansion
($2^{\prime}$) $A_{i}^{(2)}=UB_{i}^{(2)}U^{T}$ for all $i=1,\ldots,m_{2}$;
%TCIMACRO{\TeXButton{Newline}{\newline}}%
%BeginExpansion
\newline
%EndExpansion%
%TCIMACRO{\TeXButton{indent}{\indent}}%
%BeginExpansion
\indent
%EndExpansion
($3^{\prime}$) $A_{i}^{(3)}=\bar{U}B_{i}^{(3)}U^{\ast}$ for all $i=1,\ldots
,m_{3}$; and
%TCIMACRO{\TeXButton{Newline}{\newline}}%
%BeginExpansion
\newline
%EndExpansion%
%TCIMACRO{\TeXButton{indent}{\indent}}%
%BeginExpansion
\indent
%EndExpansion
($4^{\prime}$) $A_{i}^{(4)}=\bar{U}B_{i}^{(4)}U^{T}$ for all $i=1,\ldots
,m_{4}$.
\end{theorem}

Suppose that four sets of pairs of $n$-by-$n$ complex matrices are given (some
of these sets may be empty), and it is required to determine if the
simultaneous unitary equivalences stated in the General Problem are valid. Our
algorithm proceeds as follows:

\begin{algorithm}
Construct the block matrices $A$ and $B$ according to the prescription in the
preceding theorem. Then construct%
\[
K_{A}=\left[
\begin{array}
[c]{cccc}%
0 & I_n & AA^{\ast} & A\bar{A}\\
& 0 & I_n & A^{T}\bar{A}\\
&  & 0 & I_n\\
&  &  & 0
\end{array}
\right]  \text{ and }K_{B}=\left[
\begin{array}
[c]{cccc}%
0 & I_{n} & BB^{\ast} & B\bar{B}\\
& 0 & I_{n} & B^{T}\bar{B}\\
&  & 0 & I_{n}\\
&  &  & 0
\end{array}
\right]  \text{.}%
\]
The four given sets of pairs of matrices satisfy the required simultaneous
unitary equivalences if and only if $K_{A}$ and $K_{B}$ are unitarily similar.
That unitary similarity can be confirmed or refuted with finitely many
computations by using Theorem \ref{Specht}. In those computations, only words
with factor exponents at most $3$ need to be considered.
\end{algorithm}

\section{Some comments on previous work}

A criterion for simultaneous unitary similarity that reduces the problem to
one of verifying unitary similarity of a single pair of block matrices is in
Section 2.3 of the 1998 paper \cite{VVS}. We employ different block matrices
in Lemma \ref{basic} because we want it to embrace both simultaneous unitary
congruence and simultaneous unitary similarity.

The problem of simultaneous unitary similarity was also studied in the 2003
paper \cite{AI}. The criterion developed there is formally equivalent to the
finite version of our Corollary \ref{genSpecht}.

The recent paper \cite{AI2} makes the important observation that the criterion
in \cite{HornHong} can be combined with a test for simultaneous unitary
similarity to give a criterion for unitary congruence that can be verified
with finitely many computations; the authors use the test in \cite{AI}.
However, because of an extra condition imposed in \cite{AI} on matrix families
whose simultaneous unitary similarity is to be tested, the criterion in
\cite{AI2} for determining unitary congruence of two general matrices requires
assessing simultaneous unitary similarity of the \emph{four} pairs $(AA^{\ast
},BB^{\ast})$, $(A\bar{A},B\bar{B})$, $(A^{T}A^{\ast},B^{T}B^{\ast})$, and
$(A^{T}\bar{A},B^{T}\bar{B})$ rather than the \emph{three} pairs in
(\ref{basic three}).

\end{document}